\documentclass[10pt]{amsart}
\usepackage[utf8x]{inputenc}
\usepackage[toc]{appendix}
\usepackage{amssymb,amscd,mathtools ,amsmath, commath}
\usepackage{amsfonts}
\usepackage{comment, fullpage, anysize}
\usepackage{enumerate}
\marginsize{1.5in}{1.5in}{1in}{1in}

\newtheorem{theorem}{Theorem}[section]
\newtheorem{proposition}{Proposition}[section]
\newtheorem{corollary}{Corollary} [section]
\newtheorem{lemma}[theorem]{Lemma}

\theoremstyle{remark}

\newcommand{\eps}{\varepsilon}
\newcommand{\diverg}{\textup{div}\,}
\newcommand{\grad}{\textup{grad}\,}

\newcommand{\Z}{\mathbb{Z}}
\newcommand{\R}{\mathbb{R}}
\newcommand{\cC}{\mathbb{C}}

\newcommand{\Sp}{\mathbb{S}^2}
\newcommand{\sS}{\mathbb{S}}

\newcommand{\curl}{\textup{curl}\,}

\newcommand{\dist}{\textup{dist}}
\newcommand{\m}{\boldsymbol{m}}

\newcommand{\V}{\boldsymbol{V}}
\newcommand{\X}{\boldsymbol{X}}
\newcommand{\F}{\boldsymbol{F}}
\newcommand{\gG}{\boldsymbol{G}}
\newcommand{\Heff}{\boldsymbol{H}}
\newcommand{\vv}{\boldsymbol{v}}
\newcommand{\w}{\boldsymbol{w}}

\newcommand{\Id}{\textup{Id\,}}
\newcommand{\vortexvector }{\boldsymbol{a}}
\newcommand{\dd}{\boldsymbol{d}}
\newcommand{\xxi}{\boldsymbol{\xi}}
\newcommand{\jj}{\boldsymbol{j}}
\newcommand{\pp}{\boldsymbol{p}}
\newcommand{\eeta}{\boldsymbol{\eta}}
\newcommand{\nnu}{\boldsymbol{\nu}}
\newcommand{\pphi}{\boldsymbol{\phi}}
\newcommand{\vvarphi}{\boldsymbol{\varphi}}
\newcommand{\rR}{\boldsymbol{R}}
\newcommand{\Q}{\boldsymbol{Q}}

\title{Motion of the Ginzburg-Landau vortices for the mixed flow with convective forcing}
\author{Olga Chugreeva}
\address{Lehrstuhl I f\"{u}r Mathematik\\
RWTH Aachen University\\
Pontdriesch 14-16 52056 Aachen Germany}
\email{olga@math1.rwth-aachen.de}
\begin{document}
\maketitle
\begin{abstract}   
We consider the mixed Ginzburg-Landau flow that is supplemented with convective derivatives of the unknown function. We show that the associated vortex motion law is the mixed flow of the renormalized energy with new nonlinear forcing terms. These terms are uniquely determined by the extra terms in the initial PDE. Our proof relies on the assumption that the initial data are close to optimal.
\end{abstract}

\section{Introduction}\label{Intro}
We study the mixed Ginzburg-Landau flow that contains additional convective terms. The equation reads
\begin{equation}\label{GLConvectionDet}
(\lambda_{\eps}+i)\partial_{t} u_{\eps}+k_{\eps}(\F\cdot\nabla) u_{\eps}+(\gG\cdot\nabla) (iu_{\eps})=\Delta u_{\eps}+\frac{1}{\eps^{2}}(1-|u_{\eps}|^{2})u_{\eps}.
\end{equation}
The function $u_{\eps}$ is complex-valued. It is defined on a smooth, bounded, simply connected domain $D\subset \R^{2}$. The vector fields $\F(x,t)$ and $\gG(x,t)$ are smooth in both variables. The convective derivative $(\F\cdot\nabla) u_{\eps}$ is a two-dimensional vector field with the components $(\F,\nabla u_{\eps}^{j})$ for $j=1, 2$. The parameter $\eps$ is positive. The scaling factors are $k_{\eps}=1/|\log\eps|$ and $\lambda_{\eps}=\lambda_{0}\cdot k_{\eps}$, with $\lambda_{0}$ positive. We complement \eqref{GLConvectionDet} with either Dirichlet or zero Neumann boundary conditions. 

We are interested in the asymptotic behavior of the solutions to \eqref{GLConvectionDet} when $\eps$ goes to zero. The key phenomenon is the emergence of point singularities of the solution, the so-called \textit{vortices}. The behavior of $u_{\eps}$ is effectively reduced to the motion of the vortices. Our main result is the identification of the vortex equation that arises from \eqref{GLConvectionDet}. We show that the vortices obey the system of ODEs 
\begin{equation}\label{GLDetVortexMotionIntro}
(\lambda_{0}+i)\dot{a}_{k}-\F(a_{k}(t), t)-d_{k}i\gG(a_{k}(t), t)=-\frac{1}{\pi}\partial_{a_{k}}W(\vortexvector(t)).
\end{equation}
Here, $a_{k}$ is the position of the $k$th vortex and $d_{k}$ is the corresponding \textit{degree} - the topological charge carried by the vortex. The renormalized energy $W(\vortexvector)$ depends on the overall vortex configuration $\vortexvector$ and encodes the interaction between the vortices. We note that the external vector fields $\F$ and $\gG$ are transferred from the initial PDE into the effective motion law. 

The equation \eqref{GLConvectionDet} is the natural generalization of the evolution equations based on the Ginzburg-Landau energy functional 
\begin{equation}\label{GLEnergy}
E_{\eps}(u):=\int\limits_{D}\frac{1}{2}|\nabla u|^{2}+\frac{1}{4\eps^{2}}(1-|u|^{2})^{2}\, dx. 
\end{equation}
The expression on the right-hand side of \eqref{GLConvectionDet} is the negative $L^{2}$ gradient of $E_{\eps}(u_{\eps})$. The functional $E_{\eps}(u)$ is the reduced version of the energy functional that is used to describe superconductivity and superfluidity. In this regard, the vortices of $u_{\eps}$ provide a toy model of the Abrikosov vortices that are observed in type-II superconductors.

The vortex dynamics associated to various flows of the Ginzburg-Landau energy has been studied since early 1990es. It is well understood by now. The vortex motion law for the Ginzburg-Landau-Schr\"odinger equation (Hamiltonian flow of $E_{\eps}(u_{\eps})$) has been first derived by Neu \cite{Neu}. The first result of this kind for the heat (gradient) flow of $E_{\eps}(u_{\eps})$ is due to E \cite{E}. Both authors used the method of matched asymptotics. More rigorous proofs for the heat flow have been given by Lin \cite{LinHeat}, Jerrard-Soner \cite{JS_Dynamics} and Sandier-Serfaty \cite{SSGamma}. The Schr\"odinger flow has been further investigated by Colliander-Jerrard \cite{CollianderJerrard} and Lin-Xin \cite{LinXinIncompressible}. The equation on vortices for the mixed Ginzburg-Landau flow, i.e., for \eqref{GLConvectionDet} with $\F=\gG=0$, has been established by Kurzke-Melcher-Moser-Spirn \cite{KMMS} and by Miot \cite{Miot}. In all the cases, the vortex dynamics is governed by the flow of the renormalized energy~$W$. Depending on the original PDE, the flow of~$W$ is again gradient, Hamiltonian or mixed.

When adding non-variational terms to a Ginzburg-Landau equation, we are interested in perturbations that modify the vortex dynamics without destroying it. Suppose that we start with a forced flow of~$E_{\eps}(u_{\eps})$. It should produce an equation on vortices that itself is a forced flow of~$W$. We show that the forcing with convective derivatives satisfies these requirements. The requirements are further motivated by the prospect of studying \textit{stochastic} Ginzburg-Landau equations. In the stochastic setting, we are concerned with the effect of a \textit{random} external field on the vortex dynamics. First results in this direction are presented in our work \cite{MeTightness}.

The coefficient~$k_{\eps}$ in \eqref{GLConvectionDet} singles out the correct time scale for the vortex motion and the correct amplitude for the external fields. The slow motion of the vortices is characteristic for the Ginzburg-Landau heat flow. The Ginzburg-Landau-Schr\"odinger equation does not need the rescaling. We have to damp the $\F$ term, but not the $\gG$ term, to see the effect of both fields on the vortex motion. This hints that the term $k_{\eps}(\F\cdot\nabla) u_{\eps}$ is related to the damped part of the time-derivative, $\lambda_{\eps}\partial_{t} u_{\eps}$, and thus to the heat flow. In the same way, the term $(\gG\cdot\nabla) (iu_{\eps})$ is related to the Schr\"odinger flow. If we replace in \eqref{GLConvectionDet} the usual time-derivatives with the material derivatives, we arrive at the equation 
\begin{equation*}
\lambda_{\eps}D_{\F}u+ iD_{\gG}u=\Delta u_{\eps}+\frac{1}{\eps^{2}}(1-|u_{\eps}|^{2})u_{\eps}.
\end{equation*}
For an arbitrary vector field $\Heff$, $D_{\Heff}u:=\partial_{t} u+(\Heff\cdot\nabla) u$ is the material derivative defined by $\Heff$. From this point of view, \eqref{GLConvectionDet} is again a combination of the heat and the Schr\"odinger  flow. The scaling on the left-hand side of \eqref{GLConvectionDet} simply reproduces the scaling in the mixed Ginzburg-Landau flow without forcing. This leads to the richest possible vortex dynamics. The equation \eqref{GLDetVortexMotionIntro} takes into account both the interaction between the vortices and the impact of the external fields $\F$ and $\gG$. If the $\F$ term in \eqref{GLConvectionDet} is damped stronger than by $k_{\eps}$, it does not affect the vortex motion. Similarly, if the $\gG$ term is damped in any way at all, it does not appear in \eqref{GLDetVortexMotionIntro}. 

Another reason to study \eqref{GLConvectionDet} is the formal analogy between the mixed Ginzburg-Landau flow and the Landau-Lifshitz-Gilbert equation from the theory of ferromagnetism. The closest counterpart of \eqref{GLConvectionDet} is a particular variant of the latter, the Landau-Lifshitz-Slonczewski equation. It describes the evolution of the magnetization direction $\m_{\eps}:D\to\Sp$ in the presence of an external spin-current $\vv$.  In this context, $D$ represents a thin magnetic sample. The equation reads 
\begin{equation}\label{Landau-Lifshitz-Slonczewski}
\partial_{t}\m_{\eps}+(\vv\cdot\nabla) \m_{\eps}+\beta\m_{\eps}\times(\vv\cdot\nabla) \m_{\eps}=\m_{\eps}\times(\lambda_{\eps} \partial_{t}\m_{\eps}+\nabla_{L^{2}} E_{\eps}(\m_{\eps})).
\end{equation}
The basic mathematical properties of \eqref{Landau-Lifshitz-Slonczewski} have been scrutinized in \cite{MelcherPtashnyk}.

In certain situations, the magnetic energy $E_{\eps}(\m_{\eps})$ is modeled by a functional similar to the Ginzburg-Landau energy. The functional is given by
\[E_{\eps}(\m)=\int\limits_{D}\frac{1}{2}|\nabla \m|^{2}+\frac{1}{2\eps^{2}}m_{3}^{2}\, dx.
\]
Here, $m_{3}$ is the third component of the vector $\m$. For a more detailed discussion, see Section 7 in \cite{KMMS_GL_LLG}.

We can formally link equation \eqref{Landau-Lifshitz-Slonczewski} to \eqref{GLConvectionDet} in the following way. We project \eqref{Landau-Lifshitz-Slonczewski} on the tangent plane $T_{\m_{\eps}}\Sp$ that is identified with $\cC$. Then, the vector multiplication by $\m_{\eps}$ becomes the multiplication by $i$. This transforms equation \eqref{Landau-Lifshitz-Slonczewski} into equation \eqref{GLConvectionDet} with $\F=\gG=\vv$ and $k_{\eps}=\beta$. 

The solutions to \eqref{Landau-Lifshitz-Slonczewski} can develop point singularities similar to the Ginzburg-Landau vortices. These magnetic vortices move according to the Thiele equation, which is analogous to \eqref{GLDetVortexMotionIntro}. For the case $\vv=0$, the equation has been known since the works of Thiele \cite{thiele1973steady} and Huber \cite{huber1982dynamics}. It has been rigorously derived in \cite{KMMS_GL_LLG} and \cite{KMMSExcess}. The passage to the vortex motion law for \eqref{Landau-Lifshitz-Slonczewski} has been justified in \cite{KMM_Spin}. The role of the field $\vv$ is essentially the same as that of the fields $\F$ and $\gG$ in \eqref{GLDetVortexMotionIntro}. Our conclusions are therefore the Ginzburg-Landau counterpart of the results in \cite{KMM_Spin}. 

The study of the Landau-Lifshitz-Slonczewski equation is, however, more delicate than that of the Ginzburg-Landau equation. The main problem is that the solutions of the former equation are not necessarily globally smooth for a positive $\eps$. A magnetic vortex, which reflects the behavior of the whole family $(\m_{\eps})$, is not the only type of singularity. Each function $\m_{\eps}$ could independently nucleate the so-called harmonic bubbles. The question, whether the formation of bubbles really occurs in the Landau-Lifshitz equations, is open. In \cite{KMM_Spin}, the authors have to preclude the possible bubbling. Consequently, the conditions on the field $\vv$ and on the degrees of the magnetic vortices are quite strict. We do not need to adopt these assumptions for \eqref{GLConvectionDet}. We allow the degrees of both signs and we do not impose any restrictions on the structure of the external fields. Since we are not constrained by the physical considerations, the fields $\F$ and $\gG$ are unrelated to each other.

While treating the singularities of \eqref{GLConvectionDet}, our overall strategy is the same as in the works \cite{KMMS} and \cite{KMM_Spin}. We define the vortex positions as the concentration points of the Jacobian of $u_{\eps}$. To identify the vortex motion law, we use the evolution equation for the Jacobian in combination with the estimates on the energy excess. 
Both major parts of the argument have to be adjusted to the situation when the equation contains forcing terms. 

The concentration properties of the Jacobian are tied to the Ginzburg-Landau energy of the solution. The behavior of the energy is the most complicated aspect of \eqref{GLConvectionDet}. The functional $E_{\eps}(u_{\eps})$ can increase with time because the convective terms bring additional amount of energy into the system. This situation is new compared to the evolution Ginzburg-Landau equations. For them, the energy either decreases (heat and mixed flows) or is conserved (Schr\"odinger flow). 
We show that the energy can not increase too strongly. This requires a careful analysis of the impact that the convective terms have on the energy evolution. We find that the roles of the terms driven by $\F$ and $\gG$ are quite different. To quantify the effect of the first of them, it is sufficient to apply the standard integral inequalities. For the second, the usual tricks do not provide satisfactory estimates. We rather make use of the precise structure of the emerging terms.  
 
We establish the vortex motion law by, roughly speaking, passing to the limit in the evolution equation for the Jacobian. The equation includes terms that depend on the external fields. We have to comprehend the limiting behavior of these new quantities. To this end, we show that the stress tensor and the momentum of $u_{\eps}$ concentrate at the vortex cores. These facts are not well-known, but they follow readily from the concentration properties of the Jacobian. Very similar results have been used in \cite{KMM_Spin} for the Landau-Lifshitz-Slonczewski equation. When we pass to the limit, we use, at a certain point, the estimates in terms of the energy excess. For this reason, we have to start with the well-prepared initial data. The energy of our initial data is approximately minimal among all functions with the given vortex configuration.   

In \cite{TiceBoundaryCurrent} and \cite{Serfaty_Tice_pinning}, equations similar to \eqref{GLConvectionDet} appear in the context of the physically complete gauge-invariant Ginzburg-Landau theory. The equations stem from a special gauge transformation in a Ginzburg-Landau equation with magnetic field. The initial equation is, in addition, coupled with a boundary current. Exactly the boundary current produces the extra convective terms. The authors derive the effective motion law for these equations, but in regimes different from ours. In \cite{Serfaty_Tice_pinning}, the vortex dynamics is mixed, like in \eqref{GLDetVortexMotionIntro}. Still, due to the scaling, it is driven by the external vector fields alone (no $\partial_{a_{k}}W$ in the equation). The work \cite{TiceBoundaryCurrent} treats the perturbed Ginzburg-Landau heat flow. In both cases, physics prescribes Neumann boundary conditions and some structural features of the external vector fields. In particular, the fields do not depend on time. 

The article is organized as follows. Section~2 contains all preliminary information on the notation and on the Ginzburg-Landau vortices. In the same section, we formulate our theorem and comment on the existence and regularity of the solution to \eqref{GLConvectionDet}.
Section~\ref{sec:GLDetEnEst} is dedicated to the analysis of the Ginzburg-Landau energy and of the kinetic energy of the solutions. In Section~\ref{sec:GLDetJacobian}, we obtain the existence and continuity of the vortex paths. Moreover, we establish the results on the concentration. It remains to show that the vortex trajectories satisfy \eqref{GLDetVortexMotionIntro}. We prove this via the Gronwall argument in Sections \ref{sec:EnExc} and~\ref{proof_complete}. 

\section{Preliminaries}
\subsection{General notation}

We distinguish between vectors in $\R^{n}$ for $n\geqslant 2$ and scalars, either in $\R$ or in $\cC\simeq \R^{2}$. The vectors are in bold. Scalars and vector components are in light. Still,  for $u\in \cC$ we write $u=(u^{1}, u^{2})$, if we view $u$ as a vector in $\R^{2}$. 

For $u, v\in \cC$ we use the real scalar product $(u, v):=u^{1}v^{2}+u^{2}v^{1}$.
The product of $u\in \cC$ with $i$ is the rotation by $\pi/2$: $iu:=(-u^{2}, u^{1})$. 
The ``vector product'' for $u, v\in \cC$ is given by $u\times v:=(iu, v)=u^{1}v^{2}-u^{2}v^{1}$. 

The matrix $\Id$ is the identity matrix of size two.
For two matrices $A=(a_{jk})$ and $B=(b_{jk})$, $A:B$ is their Frobenius product $A:B=a_{jk}b_{jk}$.

The set $B_{s}(x)$ is the closed ball of radius $s$ centered at $x\in \R^{2}$. We denote by $\vortexvector$ an $N$-tuple $(a_{1}, ... , a_{N})$ of distinct points belonging to the domain $D$. We set \[\rho_{\vortexvector}:=\min_{k\neq l}\{\tfrac{1}{2}|a_{k}-a_{l}|, \dist(a_{k},\partial D)\}.\]  


The operator $\nabla$ is reserved for the differentiation with respect to the spatial variables. For a function $\psi:D\to \R$, $\curl \psi$ is a vector field
$\curl \psi:=(-\partial_{2}\psi, \partial_{1}\psi).$
For a vector field $\pphi:D\to \R^{2}$,  
$\curl \pphi :=-\partial_{2}\phi^{1}+ \partial_{1}\phi^{2}$ is a scalar function. For a matrix $A=(a_{jk}(x))$, $\diverg A$ is the vector with the components $\diverg A^{k}=\partial_{j}a_{jk}$. For a scalar function $\psi$, $\grad \curl \psi$ is a $2\times 2$ matrix 
\[\grad \curl \psi=\begin{pmatrix}
-\partial_{12}\psi & -\partial_{22}\psi \\
\partial_{11}\psi  & \partial_{12}\psi 
\end{pmatrix}.
\]

The space $(C^{0, \alpha})^{*}$ is the dual of the space of $\alpha$-H\"older continuous functions on $D$. The space $(C_{c}^{0, \alpha})^{*}$ is the dual of the space of $\alpha$-H\"older continuous functions that are compactly supported in $D$. 
The spaces $(C^{0, \alpha}([0, T]\times D))^{*}$ and $(C_{c}^{0, \alpha}([0, T]\times D))^{*}$ are defined in the similar way. Sometimes we shall need the elements of these spaces, and sometimes - vectors composed of such elements. We do not distinguish the two situations in our notation. 

We write $q_{\eps}=o(1)$ for a family of positive quantities $q_{\eps}$, if $\lim_{\eps \to 0} q_{\eps}=0$. In the same manner, we write $q_{\eps}\ll r_{\eps}$, if $q_{\eps}/r_{\eps}=o(1)$.

Throughout the text, $C$ is a positive constant that does not depend on $\eps$. It may change from line to line.

\subsection{Quantities related to $u_{\eps}$}

For any sufficiently smooth function \newline $u\,:\R^{2}\to\cC$, the \textit{Jacobian} of $u$ is given by
\begin{equation}\label{def:Jacobian}
J(u):=\det(\nabla u).
\end{equation}
The Jacobian is an \textit{exact} $0$-form. We have that $J(u)=\tfrac{1}{2}\,\curl \jj(u)$, where the vector 
\begin{equation}\label{def:current}
\jj(u):=(u\times\nabla u)=(u\times \partial_{1} u, u\times \partial_{2} u)
\end{equation}
is the \textit{current} of $u$.
In the same vein, for $u=u(x, t):\R^{2}\times \R_{+}\to\cC$, its \textit{total current} 
is identified with the 1-form 
$\jj_{tl}=(u\times \partial_{t} u)\,dt+(u\times \partial_{1} u )\,dx_{1}+(u\times \partial_{2} u) \,dx_{2}.$ 
Accordingly, the \textit{total Jacobian} of $u$ is the exterior derivative of $\jj_{tl}(u)$.
In other words, 
$J_{tl}(u)=V^{1}\,dt\wedge dx_{1}+V^{2}\,dt\wedge dx_{2}+J(u)\,dx_{1}\wedge dx_{2}, 
$
where $J(u)$ is the usual Jacobian \eqref{def:Jacobian} and $V^{j}=\partial_{t}(u\times u_{j})-\partial_{j}(u\times u_{t})$. We write $(J(u), \V(u))$ for the three-dimensional vector representing $J_{tl}$. The two-dimensional vector $\V(u)=(V^{1}, V^{2})$ is the \textit{velocity of the Jacobian}. Since the form $J_{tl}$ is exact, $J(u)$ and $\V(u)$ satisfy the equation
\begin{equation}\label{JacobianRelation}
\partial_{t}J(u)-\curl \V(u)=0.
\end{equation}

The \textit{momentum} of $u$ is the vector 
$
\pp(u):=(\partial_{t}u, \nabla u)=((\partial_{t}u, \partial_{1} u), (\partial_{t}u, \partial_{2} u))
$.

The \textit{stress tensor} is a matrix $(\nabla u\otimes \nabla u)$ with the entries 
$(\nabla u\otimes \nabla u )_{jk}=(\partial_{j} u,\partial_{k} u)
$.

The \textit{energy density} of $u$ is the function
\begin{equation*}
e_{\eps}(u):=\frac{1}{2}|\nabla u|^{2}+\frac{1}{4\eps^{2}}(1-|u|^{2})^{2}.
\end{equation*}
It is the integrand in the Ginzburg-Landau energy \eqref{GLEnergy}. 
The \textit{rescaled energy density} of the solution to \eqref{GLConvectionDet} is given by
\begin{equation}\label{def:mu}
\mu_{\eps}(t):=k_{\eps}e_{\eps}(u_{\eps}(t)).
\end{equation}

\subsection{The vortices and the renormalized energy}\label{sec:RenEn}

The content of this section is standard and is taken from \cite{BBH} and \cite{JerrardSoner}. It mostly concerns the static case.

For a function $u:D\to\cC$, a \textit{vortex} is an isolated zero around which $u/|u|$ has a nontrivial winding number. The winding number is called \textit{degree}. For $\eps$ going to $0$, the limiting vortices of the family $(u_{\eps})$ are distinct points $a_k\in D$ with degrees $d_k\in \Z$. The vortices of high degrees are believed to be unstable. We restrict our discussion to the case $d_k\in\{\pm 1\}$. 

We complement our equation with either Dirichlet or zero Neumann boundary conditions. In the Dirichlet case 
\begin{equation}\label{GLDetBC}
u_{\eps}\big|_{\partial D}=g(x),
\end{equation}
the function $g\in C^{\infty}(\partial D,\sS^{1})$ has a nontrivial degree. This imposes a topological constraint on the solutions of \eqref{GLConvectionDet}. A function $u:D\to\cC$ that coincides with $g$ on $\partial D$ \textit{must} have vortices in the interior of $D$.
In the Neumann case 
\begin{equation}\label{GLDetBCNeumann}
\partial_{\nnu}u_{\eps}\big|_{\partial D}=0,
\end{equation}
we assume that the initial data does have vortices. Moreover, we always assume that the initial data 
\begin{equation}\label{GLDetIC}
u_{\eps}(x,0)=u_{\eps}^{0}(x)
\end{equation}
is smooth. 

A vortex carries an unbounded amount of energy that is proportional to $\log(1/\eps)$.
Therefore, the Ginzburg-Landau energy of a function with vortices diverges as $\eps$ goes to $0$.
The inequality
\begin{equation}\label{upper_bound}
E_\eps(u_\eps)\leqslant C \log(1/\eps)
\end{equation} 
provides a non-sharp upper bound on the actual number of vortices. If \eqref{upper_bound} holds, the Jacobians of $u_{\eps}$ are precompact in the strong topology of $(C^{0,\alpha})^{*}$ for every $\alpha>0$. They concentrate at the vortex sites \cite{JerrardSoner}:
\begin{equation*}
J(u_{\eps})\rightarrow \pi\sum\limits_{k=1}^{N}d_{k}\delta_{a_{k}} \text{ with } \pi\sum\limits_{k=1}^{N}|d_{k}|\leqslant C.
\end{equation*}

A \textit{vortex configuration} $(\vortexvector,\dd)$ is the collection of vortex positions $\vortexvector$ together with the collection of corresponding degrees $\dd\in (\pm 1)^{N}$. We say that a family $(u_\eps)$ \textit{exhibits the vortex configuration} $(\vortexvector,\dd)$, if $J(u_{\eps})\rightarrow \pi\sum_{k=1}^{N}d_{k}\delta_{a_{k}}.$

For a vortex configuration $(\vortexvector, \dd)$, the associated \textit{canonical harmonic map} $u_{*}^{(\vortexvector, \dd)}$ is an $\sS^{1}$-valued map that solves the system
\begin{equation}\label{CHM}
\left\{\begin{array}{l}
\diverg \jj(u_{*})=0,\\
\curl \jj(u_{*})=2\pi\sum_{k=1}^{N}d_{k}\delta_{a_{k}}.
\end{array}
\right. 
\end{equation}
The boundary conditions are $u_{*}=g$ in the Dirichlet case and $(\nnu, \jj(u_{*}))=0$ in the Neumann case. In the former situation, the canonical harmonic map is defined uniquely, in the latter - up to a constant. 

The \textit{renormalized energy} $W$ of the vortex configuration $(\vortexvector, \dd)$ is given by 
\begin{equation}
W(\vortexvector, \dd):=\lim\limits_{s\to 0}\Big(\int_{D\setminus \mathop{\cup}_{k=1}^{N}B_{s}(a_{k})}|\nabla u_{*}|^{2}\, dx-\pi\sum\limits_{k=1}^{N}d_{k}^{2}\log(1/s)\Big),
\end{equation}
with $u_{*}=u_{*}^{(\vortexvector, \dd)}$. We shall sometimes write $W(\vortexvector)$ instead of $W(\vortexvector, \dd)$, since in our situation the parameters $\dd$ are fixed. In both Dirichlet and Neumann case, the renormalized energy has the form
$W(\vortexvector)=-\pi\sum_{j\neq k}d_{k}d_{j}\log|a_{k}-a_{j}|+\textit{boundary terms}$. 
This shows that the vortices interact like charged particles, with the degree playing the role of the charge. The renormalized energy blows up if two vortices bump into each other or a vortex hits the boundary of $D$. 

We denote by $\partial_{a_{k}}W(\vortexvector)$ the vector $(\partial_{a_{k}^{1}}W(\vortexvector), \partial_{a_{k}^{2}}W(\vortexvector))$. The identity below links the stress tensor of the canonical harmonic map to the gradient of the renormalized energy.
\begin{proposition}[\cite{KMMS} Proposition 3.2; \cite{JerrardSpirnRefined}, Lemma 8]\label{prop:GradOfW}
Suppose that a smooth function $\varphi$ 
is affine in $B_{s}(a_{k})$ for some $0<s<\rho(\vortexvector)$ and a fixed $k\in\{1,..., N\}$. Suppose moreover that $\varphi=0$ in a neighborhood of $a_{j}$ for $j\neq k$. Let $u_{*}$ be the canonical harmonic map associated to $(\vortexvector, \dd)$. Then, 
there holds
\begin{equation}\label{GradientOfRenEn}
\int\limits_{D}(\grad \curl \varphi) :(\jj(u_{*})\otimes \jj(u_{*})) \,dx=-\curl \varphi(a_{k})\cdot \partial_{a_{k}}W(\vortexvector, \dd).
\end{equation}
\end{proposition} 

The Ginzburg-Landau energy and the renormalized energy are related as follows.
\begin{proposition}
Suppose that $(u_{\eps})$ exhibits a vortex configuration $(\vortexvector, \dd)$ with \linebreak$\dd\in(\pm 1)^N$. Then there exists a constant $\gamma>0$ such that
\begin{equation}\label{GammaConv}
\liminf\limits_{\eps\to 0}(E_{\eps}(u_{\eps})-\pi N \log(1/\eps))\geqslant N\gamma + W(\vortexvector, \dd).
\end{equation}
\end{proposition}
The generic constant $\gamma$ is the finite part of the energy of the vortex core. 

The \textit{approximate energy} $W_{\eps}(\vortexvector, \dd)$ is given by 
\begin{equation*}
W_{\eps}(\vortexvector, \dd):=\pi N\log(1/\eps)+N\gamma +W(\vortexvector, \dd).
\end {equation*}
It is the minimal Ginzburg-Landau energy that a family $(u_{\eps})$ with a vortex configuration $(\vortexvector, \dd)$ can have. 
For such a family, the \textit{energy excess} $D_{\eps}$ is defined by
\begin{equation}\label{eq:EnergyExcess}
D_{\eps}:=E_{\eps}(u_{\eps})-W_{\eps}(\vortexvector, \dd).
\end{equation}
If $(u_{\eps})$ depends on time, so does $D_{\eps}$.

The family $(u_{\eps})$ is \textit{well-prepared}, if for it $D_{\eps}=o(1)$. 

\subsection{Requirements on the external vector fields}
The vector fields $\F(x,t)$ and $\gG(x, t)$ are smooth in both variables and satisfy
\[\F(x,0)\big|_{\partial D}=\gG(x, 0)\big|_{\partial D}=0.
\]
Then, the compatibility conditions for the initial-boundary value problems \eqref{GLConvectionDet} - \eqref{GLDetBC} - \eqref{GLDetIC} and \eqref{GLConvectionDet} - \eqref{GLDetBCNeumann} - \eqref{GLDetIC} are fulfilled.
In the Neumann case, we need in addition the equality
\begin{equation}\label{CondOnG}
(\gG(x,t), \nnu)|_{\partial D}=0
\end{equation}
to hold for all $t\geqslant 0$. Here, $\nnu$ is the unit outer normal on $\partial D$.

\subsection{Statement of the main result}\label{sec:MainResult}
The aim of the present work is to prove the theorem below.
\begin{theorem}\label{ThmGLDet}
Let $(u_{\eps})$ be the family of solutions to \eqref{GLConvectionDet} with either Dirichlet or Neumann boundary conditions. Suppose that the initial data $u_{\eps}^{0}$ are well-prepared. In other words, there exists a configuration $(\vortexvector^{0}, \dd)$ of $N$ vortices with $\dd\in(\pm 1)^N$ such that 
\begin{equation}\label{EnergyAtTimeZero}
E_{\eps}(u_{\eps}^{0})=\pi N \log (1/\eps)+N\gamma +W(\vortexvector^{0},\dd)+o(1)
\end{equation}
and
\begin{equation}\label{JacobianAtTimeZero}
J(u_{\eps}^{0})\xrightarrow[\eps\to 0]{} \pi\sum_{k=1}^{N}d_{k}\delta_{a^{0}_{k}} \text{ in the space } (C_{c}^{0,1})^{*}.
\end{equation}
Then, there exists a time $T^{*}>0$ such that 
for $t\in[0, T^{*})$, the following holds.
\begin{enumerate}
\item The family $(u_{\eps}(t))$ is well-prepared;
\item the Jacobian $J(u_{\eps}(t))$ converges in $(C_{c}^{0,1})^{*}$ to the measure $\pi\sum_{k=1}^{N}d_{k}\delta_{a_{k}(t)}$;
\item the vortex trajectories $a_{k}(t)$ solve the system of ODEs
\begin{equation}\label{GLDetVortexMotion}
\left\{\begin{array}{l}
(\lambda_{0}+i)\,\dot{a}_{k}=-\frac{1}{\pi}\partial_{a_{k}}W(\vortexvector(t), \dd)+\F(a_{k}(t), t)+d_{k}i\gG(a_{k}(t), t),\\
a_{k}(0)=a_{k}^{0}.
\end{array}
\right.
\end{equation}
\end{enumerate}
In the Dirichlet case, two vortices collide at time $T^{*}$. In the Neumann case, at $T^{*}$ either two vortices collide or a vortex hits the boundary of the domain.
\end{theorem}

\subsection{On the existence of the solution to the equation \eqref{GLConvectionDet}}\label{sec:DetGLExistence}

First, we can always find initial data satisfying the assumptions of Theorem~\ref{ThmGLDet} - see \cite{JerrardSpirnRefined}, Lemma~14. 
Second, the following proposition holds true.
\begin{proposition}\label{prop:Existence}
For every fixed $\eps>0$, each of the problems \eqref{GLConvectionDet} - \eqref{GLDetBC}- \eqref{GLDetIC} and \eqref{GLConvectionDet} - \eqref{GLDetBCNeumann}- \eqref{GLDetIC} has a unique global classical solution. 
\end{proposition}
\begin{proof}
The reasoning is the same as in \cite{KMMS}, Remark 2.4. The equation \eqref{GLConvectionDet} is equivalent to a strictly parabolic system of semilinear equations on the real functions $u_{\eps}^{1}$ and $u_{\eps}^{2}$. For such a system, we construct a weak solution via the Galerkin approximation. We then improve the regularity with the help of the Calderon-Zygmund estimates \cite{Taylor_III}. 
\qquad\end{proof}

Even though the system on $u_{\eps}^{1}$ and $u_{\eps}^{2}$ is parabolic, the weak maximum principle does not hold for it. This is due to the strong coupling. Hence, the uniform estimate $|u_{\eps}|\leqslant 1$ is \textit{not} available, even in the Dirichlet case. The lack of this additional information is the only reason to use the excess energy estimates. Otherwise we could employ a more robust method of \cite{KMMSExcess} for the derivation of the effective equation. This would allow us to lift the assumption on the well-preparedness of the initial data.

\subsection{Conservation laws}\label{sec:ConsLaws}  

Here, we derive several intermediate evolution equations on the quantities defined in Section~\ref{sec:RenEn}. 

For the field $(\gG\cdot \nabla)iu_{\eps}$, we easily find that 
\begin{equation}\label{GWithNablaU}
\big((\gG\cdot \nabla)iu_{\eps}, \nabla u_{\eps}\big)=i\gG\cdot J(u_{\eps})
\end{equation}
and
\begin{equation}\label{GWithTimeDerivative}
\big((\gG\cdot \nabla)iu_{\eps}, \partial_{t} u_{\eps}\big)=-(\gG , \V(u_{\eps})).
\end{equation}

We obtain \textit{the conservation law for the energy density} by multiplying both sides of \eqref{GLConvectionDet} with $\partial_{t} u_{\eps}$ and using \eqref{GWithTimeDerivative}. The law is given by
\begin{equation}\label{ConsOfEn}
\partial_{t}e_{\eps}(u_{\eps})=-\lambda_{\eps}|\partial_{t}u_{\eps}|^{2}-k_{\eps}(\F, \pp( u_{\eps}))+(\gG, \V(u_{\eps})) +\diverg \pp(u_{\eps}).
\end{equation}

We also need \textit{the conservation law for the Jacobian}. 
We first use the identity \eqref{JacobianRelation} to see that 
$\partial_{t}J(u_{\eps})=\curl (\partial_{t}u_{\eps}\times \nabla u_{\eps})=\curl (i\partial_{t}u_{\eps}, \nabla u_{\eps}).
$
We then substitute $i \partial_{t} u_{\eps}$ from \eqref{GLConvectionDet}. Several formal transformations and \eqref{GWithNablaU} finally yield 
\begin{multline}\label{ConsOfJacobian}
\partial_{t}J(u_{\eps})+\lambda_{\eps}\curl \pp(u_{\eps})=\curl \diverg (\nabla u_{\eps}\otimes \nabla u_{\eps})\\
-k_{\eps}\curl(\F\cdot (\nabla u_{\eps}\otimes \nabla u_{\eps}))-\curl((i\gG)\cdot J(u_{\eps})).
\end{multline}
 
The \textit{conservation law for the mass} reads
\begin{align}\label{ConsOfMass}
\partial_{t}(\tfrac{1-|u_{\eps}|^{2}}{2})+\lambda_{\eps}(u_{\eps}\times \partial_{t} u_{\eps})+k_{\eps}(\F,\jj(u_{\eps}))+(\gG, \nabla (\tfrac{1-|u_{\eps}|^{2}}{2}))=\diverg \jj(u_{\eps}).
\end{align}
We arrive at it when we multiply \eqref{GLConvectionDet} with $iu_{\eps}$. 

\section{Energy estimates and corollaries}\label{sec:GLDetEnEst}

In this section, we study the Ginzburg-Landau energy and the kinetic energy of the solution to \eqref{GLConvectionDet}. We first prove a rather crude estimate on the Ginzburg-Landau energy in Lemma~\ref{Lem:GLEn}. From it, we deduce an optimal estimate on the kinetic energy in Proposition~\ref{prop:kinetic_energy}. This implies the relative compactness of the total Jacobian in Corollary~\ref{cor:CompOfTheTotalJacobian}. We finally improve the control on the Ginzburg-Landau energy in Proposition~\ref{prop:GLEnergyUpperBound}.

\begin{lemma}\label{Lem:GLEn}
For every $T>0$, there exists a constant $C>0$ such that for all $0\leqslant t\leqslant T$, there holds
\begin{equation}\label{GLEnergyUpperBoundCrude}
E_{\eps}(u_{\eps}(t))\leqslant C\log(1/\eps).
\end{equation}
\end{lemma}
\begin{proof}
We integrate \eqref{ConsOfEn} over the set $D\times[0, T]$. Then, we obtain 
\begin{align}\label{eq:EnergyEvolution}
E_{\eps}(u_{\eps}(T))+\lambda_{\eps}\int\limits_{0}^{T}&\int\limits_{D}|\partial_{t}u_{\eps}|^{2}
= E_{\eps}(u_{\eps}^{0})
-k_{\eps}\int\limits_{0}^{T}\int\limits_{D}(\F, \pp(u_{\eps}))+\int\limits_{0}^{T}\int\limits_{D}(\gG,\V(u_{\eps})).
\end{align}
We estimate the integrals on the right-hand side of \eqref {eq:EnergyEvolution}. To the first of them, we apply successively the Cauchy-Schwartz inequality and the Young inequality with appropriately chosen coefficients. This gives
\begin{multline}\label{EstOnF}
\Big|k_{\eps}\int\limits_{0}^{T}\int\limits_{D}(\F, \pp(u_{\eps}))\Big|
\leqslant C\Big(\lambda_{\eps}\int\limits_{0}^{T}\int\limits_{D}|\partial_{t}u_{\eps}|^{2}\Big)^{\tfrac{1}{2}}\cdot\Big(k_{\eps}\int\limits_{0}^{T}\int\limits_{D}|\nabla u_{\eps}|^{2}\Big)^{\tfrac{1}{2}}\\
\leqslant \frac{\lambda_{\eps}}{4}\int\limits_{0}^{T}\int\limits_{D}|\partial_{t}u_{\eps}|^{2}+ Ck_{\eps} \int\limits_{0}^{T}E_{\eps}(u_{\eps}(t))\,dt.
\end{multline}

The estimate for the second integral is slightly more involved. The decisive step is to use the structure of $\V(u_{\eps})$ first. The remaining estimates are again standard. We start with integration by parts
\begin{multline*}
\int\limits_{0}^{T}\int\limits_{D}(\gG,\V(u_{\eps}))
=-\frac{1}{2}\int\limits_{0}^{T}\int\limits_{D}\diverg\gG\cdot(u_{\eps}\times \partial_{t}u_{\eps})+\frac{1}{2}\int\limits_{0}^{T}\int\limits_{D}(\partial_{t}\gG, \jj(u_{\eps}))\\
+\frac{1}{2}\int\limits_{D}(\gG(0),\jj(u_{\eps}^{0}))-(\gG(T),\jj(u_{\eps}(T)))\,dx.
\end{multline*}
The integral over $\partial D$, which appears when we integrate by parts with respect to $x$, is equal to zero. In the Dirichlet case, this follows from the fact that the boundary data $g(x)$ is independent of time. In the Neumann case, this follows from the additional assumption \eqref{CondOnG} on the field $\gG$. 

We work with the three new terms. We take into account the bound
\begin{equation}\label{Modulus}
\int\limits_{0}^{T}\int\limits_{D}|u_{\eps}|^{2}=\int\limits_{0}^{T}\int\limits_{D}(|u_{\eps}(x,t)|^{2}-1)+T\mathcal{L}^{2}(D)\leqslant C+C\eps^{2}\int\limits_{0}^{T}E_{\eps}(u_{\eps}(t))\,dt.
\end{equation}
Here, $\mathcal L^{2}$ is the Lebesgue measure on $\R^{2}$. 
We obtain, by using the Cauchy-Schwartz inequality, \eqref{Modulus}, and the Young inequality, the estimates
\begin{multline*}
\Big|\int\limits_{0}^{T}\int\limits_{D}\diverg\gG\cdot(u_{\eps}\times \partial_{t}u_{\eps})\Big|\\
\leqslant \frac{\lambda_{\eps}}{4}\int\limits_{0}^{T}\int\limits_{D}|\partial_{t}u_{\eps}|^{2}+C\eps^{2}\log(1/\eps)\cdot\int\limits_{0}^{T}E_{\eps}(u_{\eps}(t))\,dt+C\log(1/\eps),
\end{multline*}

\begin{multline*}
\Big|\int\limits_{0}^{T}\int\limits_{D}(\partial_{t}\gG, \jj(u_{\eps}))\Big|\leqslant C \Big(\int\limits_{0}^{T}\int\limits_{D}|u_{\eps}|^{2}\Big)^{\tfrac{1}{2}}\cdot \Big(\int\limits_{0}^{T}\int\limits_{D}|\nabla u_{\eps}(x,t)|^{2}\Big)^{\tfrac{1}{2}}\\
\leqslant Ck_{\eps}\int\limits_{0}^{T}E_{\eps}(u_{\eps}(t))\, dt+C\eps^{2}\log(1/\eps)\int\limits_{0}^{T}E_{\eps}(u_{\eps}(t))\,dt+C\log(1/\eps),
\end{multline*}
and
\begin{multline*}\Big|\int\limits_{D}(\gG(0),\jj(u_{\eps}^{0}))-(\gG(T),\jj(u_{\eps}(T)))\,dx\Big|\\
\leqslant C\log(1/\eps)+C(\eps^{2}\log(1/\eps)+k_{\eps})E_{\eps}(u_{\eps}(T)).
\end{multline*}

Together with \eqref{EstOnF}, this gives
\begin{multline}\label{EnIneqPrelim}
E_{\eps}(u_{\eps}(T))+\frac{\lambda_{\eps}}{2}\int\limits_{0}^{T}\int\limits_{D}|\partial_{t}u_{\eps}|^{2}\leqslant E_{\eps}(u_{\eps}^{0})+C(\eps^{2}\log(1/\eps)+k_{\eps})E_{\eps}(u_{\eps}(T))\\
+C(\eps^{2}\log(1/\eps)+k_{\eps})\int\limits_{0}^{T}E_{\eps}(u_{\eps}(t))\,dt+C\log(1/\eps).
\end{multline}
By choosing $\eps$ small enough, we absorb the term $C(\eps^{2}\log(1/\eps)+k_{\eps})E_{\eps}(u_{\eps}(T))$ on the left-hand side. 
Now, the Gronwall Lemma and the well-preparedness of the initial data imply \eqref{GLEnergyUpperBoundCrude}.
\qquad\end{proof}

\begin{proposition}\label{prop:kinetic_energy}
There exists a time $\bar{T}>0$ and a constant $C>0$ such that 
\begin{equation}\label{kinetic_energy}
k_{\eps}\int\limits_{0}^{\bar{T}}\int\limits_{D}|\partial_{t}u_{\eps}|^{2}dx dt\leqslant C.
\end{equation}
\end{proposition}
\begin{proof}
Our proof is by contradiction. It is a modification of the argument in Lemma III.1 of \cite{SSGamma}. We assume that the claim is false. Then we rescale time in such a way that the vortices do not move, but the energy is dissipated strongly enough. This contradicts the well-preparedness of the initial data \eqref{EnergyAtTimeZero}.

Suppose that the statement of the proposition does not hold. 
Then there exist sequences $t_{n}\to 0$ and $\eps_{n}\to0$ such that  
\[k_{\eps_{n}}\int\limits_{0}^{t_{n}}\int\limits_{D}|\partial_{t}u_{\eps_{n}}|^{2}=1.
\]
For the rest of this proof, we write $u_{n}$ instead of $u_{\eps_{n}}$ and $k_{n}$ instead of $k_{\eps_{n}}$.

We first rescale time by $k _n$ and define the functions $v_{n}(x,t):=u_{n}(x, k_n t)$.
For them, our assumption gives the equality
\begin{equation}\label{RescaledKineticEn1}
\int\limits_{0}^{\tau_{n}}\int\limits_{D}|\partial_{t}v_{n}|^{2}=1,
\end{equation}
with $\tau_{n}:=t_{n}\log(1/\eps_{n})$. 
We rescale time once again by $\tau_{n}$ and define the functions $w_{n}(x,t):=v_{n}(x,\tau_{n}t)$.
Note that $w_{n}(x,t)=u_{n}(x,t_{n}t)$. For $w_{n}$, there holds
\begin{equation}\label{RescaledKineticEn2}
\int\limits_{0}^{1}\int\limits_{D}|\partial_{t}w_{n}|^{2}=\tau_{n}\ll \log(1/\eps_{n})
\end{equation}
and 
\begin{equation}\label{RescaledGLEn2}
E_{\eps}(w_{n}(t))=E_{\eps_{n}}(u_{n}(t_{n}t))\leqslant C\log(1/\eps_{n}).
\end{equation}

Due to Theorem~3 in \cite{SSProd}, the estimates \eqref{RescaledKineticEn2} and \eqref{RescaledGLEn2} imply that the total Jacobian of $w_{n}$ is relatively compact. Hence, the limits $\lim_{{n}\to \infty}J(w_{n}(t))=J(t)$ and $\lim_{n\to \infty}\V(w_{n}(t))=\V(t)$
exist in $(C^{0, \alpha})^{*}$. For them, we have $\partial_{t}J(t)-\curl \V(t)=0$ for every $t\in[0,1]$.
In other words, the vorticity of $w_{n}$ is transported by $\V$. Moreover, $\V$ satisfies the product estimate from \cite{SSProd}
\begin{equation}\label{eq:ProdEst}
\Big| \int\limits_{0}^{1}\int\limits_{D} \varphi\V\cdot\X \Big|\leqslant \liminf\limits_{n\to \infty} k_{n}\Big(\int\limits_{0}^{1}\int\limits_{D}|\X\cdot \nabla w_{n}|^{2}\Big)^{\tfrac{1}{2}}\cdot \Big(\int\limits_{0}^{1}\int\limits_{D}\varphi^{2}|\partial_{t}w_{n}|^{2}\Big)^{\tfrac{1}{2}}.
\end{equation}
The inequality holds for every H\"older continuous scalar function $\varphi$ and every H\"older continuous vector field $\X$. The right-hand side of \eqref{eq:ProdEst} is zero, by \eqref{RescaledKineticEn2} and \eqref{RescaledGLEn2}. Consequently, $\V=0$, and the vortices of $w_{n}$ do not move on the time interval $[0,1]$. They are the same as the vortices of $u_{\eps}^{0}$, and $J(w_{n}(t))\rightarrow \pi\sum_{k=1}^{N}d_{k}\delta_{a^{0}_{k}}$, for every $t\in [0,1]$. 

We return to $v_{n}$. Its vortices do not move until the time $\tau_{n}$. With \eqref{GammaConv}, we have
$E_{\eps_{n}}(v_{n}(\tau_{n}))\geqslant \pi N \log(1/\eps)+N\gamma +W(\vortexvector^{0},\dd)+o(1)= E_{\eps}(u_{\eps}^{0})+o(1)$.
Moreover, 
\begin{equation}\label{RescaledGLEn1}
E_{\eps_{n}}(v_{\eps}(\tau_{\eps}))=E_{\eps}(u_{\eps}(t_{\eps}))\leqslant C\log(1/\eps),
\end{equation}
by Lemma \ref{Lem:GLEn}. 

The functions $v_{n}$ solve an equation of the same form as \eqref{GLConvectionDet}. The only difference is the scaling of the time-derivative. 
Therefore, as in the proof of the Lemma \ref{Lem:GLEn}, we find that 
\begin{multline}\label{eq:Energy evolutionRescaled}
\lambda_{0}\int\limits_{0}^{\tau_{n}}\int\limits_{D}|\partial_{t}v_{n}|^{2}=E_{\eps_{n}}(u_{n}^{0})-E_{\eps_{n}}(v_{n}(\tau_{n}))\\-k_{n}\int\limits_{0}^{\tau_{n}}\int\limits_{D}(\F(k_{n}t),\pp(v_{n}))
+\int\limits_{0}^{\tau_{n}}\int\limits_{D}(\gG(k_{n}t),\V(v_{n})).
\end{multline}
The quantity on the right-hand side in the first line is either positive of order $o(1)$ or negative. Below we show that both integrals in the second line are of order $o(1)$. This leads to a contradiction, which concludes the proof. 

For the first of the integrals we use the Cauchy-Schwartz inequality, \eqref {RescaledKineticEn1} and \eqref{RescaledGLEn1} to obtain
\[\Big| k_{n}\int\limits_{0}^{\tau_{n}}\int\limits_{D}(\F(k_{n}t),\pp(v_{n}))\Big|
\leqslant Ck_{n}\Big(\int\limits_{0}^{\tau_{n}}E_{\eps_{n}}(v_{n}(t))\Big)^{\tfrac{1}{2}}\leqslant 
Ct_{n}^{1/2}=o(1).
\]

In the second integral, we switch to the time-scale of $w_{n}$ and get 
\[\int\limits_{0}^{\tau_{n}}\int\limits_{D}(\gG(k_{n}t),\V(v_{n}))=\int\limits_{0}^{1}\int\limits_{D}(\gG(t_{n}t),\V(w_{n})).
\]
As we already know, $\V(w_{n}(t))\to 0$ for $t\in [0,1]$. Then, this integral is of order $o(1)$.
\qquad\end{proof}

The most important consequence of the energy estimates established so far is
\begin{corollary}\label{cor:CompOfTheTotalJacobian}
The total Jacobian $(J(u_{\eps}), \V(u_{\eps}))$ of $u_{\eps}$ is relatively compact in the space $(C^{0,\alpha}([0,\bar{T}]\times D))^{*}$, for every $\alpha \in(0, 1]$.
\end{corollary}
\begin{proof}
The claim follows directly from \eqref{GLEnergyUpperBoundCrude} and \eqref{kinetic_energy} by the virtue of Theorem~5.1 in \cite{JerrardSoner}. Even though the Theorem~is formulated for the space $(C^{0,\alpha}_{c}([0,T]\times D))^{*}$, the proof works for the smaller space $(C^{0,\alpha}([0,T]\times D))^{*}$ as well. See a related discussion in the proof of Theorem~3 in \cite{SSProd}.
\qquad\end{proof}

\begin{proposition}\label{prop:GLEnergyUpperBound}
There exists a constant $C>0$ such that, for all $t\in[0, \bar{T}]$, there holds
\begin{equation}\label{GLEnergyUpperBound}
E_{\eps}(u_{\eps}(t))\leqslant \pi N \log(1/\eps)+C.
\end{equation}
\end{proposition}
\begin{proof}
We return to \eqref{eq:EnergyEvolution} and set $T=\bar{T}$.
Now we can estimate the integrals containing $\F$ and $\gG$ much more precisely. With \eqref{GLEnergyUpperBoundCrude} and \eqref{kinetic_energy}, we obtain that
\begin{equation*}
\Big|k_{\eps}\int\limits_{0}^{\bar{T}}\int\limits_{D}(\F, \pp(u_{\eps}))\Big|\leqslant C\Big(k_{\eps}\int\limits_{0}^{\bar{T}}\int\limits_{D}|\partial_{t}u_{\eps}|^{2} \Big)^{\tfrac{1}{2}}\cdot\Big(k_{\eps}\int\limits_{0}^{\bar{T}}\int\limits_{D}|\nabla u_{\eps}|^{2}\Big)^{\tfrac{1}{2}}\leqslant C.
\end{equation*}

For the second integral, we use Corollary \ref{cor:CompOfTheTotalJacobian}. Since $\V(u_{\eps})$ is relatively compact and $\gG$ is smooth, 
there exists a constant $C$ such that
\begin{equation}\label{eq:EstimateOnVG}
\Big|\int\limits_{0}^{\bar{T}}\int\limits_{D}(\gG,\V(u_{\eps}))\Big|\leqslant C.
\end{equation}
This gives the inequality $E_{\eps}(u_{\eps}(\bar{T}))\leqslant E_{\eps}(u_{\eps}^{0})+C$. Now, \eqref{GLEnergyUpperBound} is actually a consequence of \eqref{EnergyAtTimeZero}.
\qquad\end{proof}

\section{Vortex trajectories and concentration}\label{sec:GLDetJacobian}
The vortices are the concentration points of the Jacobian. From Corollary~\ref{cor:CompOfTheTotalJacobian}, we deduce in Proposition~\ref{prop:Compactness of the Jacobian} the existence and continuity of the vortex paths. We study the concentration properties of the stress tensor and the momentum of $u_{\eps}$ in Propositions \ref{prop:convergence of the stress tensor} and~\ref{prop:convergence of the momentum}. Proposition~\ref{DivCurrent} is a technical fact, which we need in the next section. It follows from the energy bounds.

\begin{proposition}\label{prop:Compactness of the Jacobian}
Let $\bar{T}$ be the time found in Proposition \ref{prop:kinetic_energy}. For every sequence \newline $\eps_{n}\to 0$, the sequence $(J(u_{\eps_{n}}),\V(u_{\eps_{n}}))$ converges, modulo a subsequence, in the space $(C^{0,\alpha}([0,\bar{T}]\times D))^{*}$. For the limit $(J(t), \V(t))$, we have
\[J(t)=\pi\sum\limits_{k=1}^{N}d_{k}\delta_{\xi_{k}(t)} \text{ and } 
\V(t)=-\pi i\sum\limits_{k=1}^{N}d_{k}\dot{\xi}_{k}(t)\delta_{\xi_{k}(t)}.
\]
The functions $\xi_{k}(t):[0,\bar{T}]\to \R^{2}$ belong to the space $H^{1}(0,\bar{T};D)$. The coefficients $d_{k}$ are integers, with $d_{k}\in \{\pm1\}$.
\end{proposition}
\begin{proof}
The existence of a convergent subsequence is given by Corollary \ref{cor:CompOfTheTotalJacobian}. The measure $J(t)$ is, for every $t$, a weighted sum of deltas by Theorem~3.1 of \cite{JerrardSoner}. Moreover, Theorem~3 of \cite{SSProd} yields that $J(t)$ belongs to the space $C^{0,1/2}([0,\bar{T}], (C^{0,1})^{*})$. The regularity of the functions $\xi_{k}$ is determined in \cite{SSGamma}, Propositions III.1 and III.2. The last two facts also imply that the coefficients $d_{k}$ are constant. Finally, the form of $\V(t)$ is completely specified by the representation for $J(t)$ via \eqref{JacobianRelation}.
\qquad\end{proof}

In addition, we have a mobility bound on the functions $\xi_{k}(t)$.
\begin{proposition}[\cite{SSProd}, Corollary 7 ]\label{prop: mobility bound}
Let $\xi_{k}(t)$ be the functions found in Proposition \ref{prop:Compactness of the Jacobian}. Then, for all $0\leqslant t_{1}<t_{2}\leqslant \bar{T}$, there holds 
\begin{equation}\label{eq:mobility bound}
\pi \sum\limits_{k=1}^{N}\int\limits_{t_{1}}^{t_{2}}|\dot{\xi}_{k}|^{2}\,dt \leqslant  \liminf\limits_{\eps\to 0}k_{\eps}\int\limits_{t_{1}}^{t_{2}}\int\limits_{D}|\partial_{t}u_{\eps}|^{2}.
\end{equation}
\end{proposition}

We use the part of Proposition~\ref{prop:Compactness of the Jacobian} that concerns $\V(u_{\eps})$ in a different, but equivalent, form. Namely, we have that for all $0\leqslant t_{1}<t_{2}\leqslant \bar{T}$ and all smooth vector fields $\w:D\times [0, \bar{T}]\rightarrow \R^{2}$, there holds
\begin{equation}\label{eq:ConvOfVelocity}
\lim\limits_{\eps\to 0}\int\limits_{t_{1}}^{t_{2}}\int\limits_{D} (\V(u_{\eps}(t)), \w)=\pi\sum\limits_{k=1}^{N}d_{k}\int\limits_{t_{1}}^{t_{2}}(\dot{\xi}_{k}(t),i\w({\xi}_{k}(t),t))\,dt.
\end{equation}

The set of paths $\xxi(t)$ coincides with the solution set $\vortexvector(t)$ of \eqref{GLDetVortexMotion} at the time $t=0$. This is due to the condition \eqref{JacobianAtTimeZero} on the initial data. Our aim is to prove that the two sets coincide for all $t\in[0, T^{*})$. We point out that $\xxi(t)$ is produced by some sequence $\eps_{n}\to 0$. Up to the end of Proposition \ref{prop:short-time}, we use the notation $\eps\to 0$ for this particular sequence. Eventually, Proposition \ref{prop:short-time} and Theorem~\ref{ThmGLDet} imply that the family $\xxi(t)$ is independent of the sequence $\eps_{n}\to 0$.
 
In view of \eqref{eq:mobility bound}, we may assume that $\bar{T}$ is smaller than the collision time $T^{*}$. Then, by Proposition \ref{prop:GLEnergyUpperBound} and the definition of the energy excess \eqref{eq:EnergyExcess}, we find that 
\begin{equation}\label{eq:energy excess of order one}
D_{\eps}(t)\ll \log(1/\eps)
\end{equation}
for every $t\in[0,\bar{T}]$.

We now deduce from Proposition \ref {prop:Compactness of the Jacobian} the remaining results on the concentration. 

\begin{proposition}[Concentration of the rescaled energy density]\label{prop:convergence of the energy density}
Suppose that $J(u_{\eps}(t))$ converges as described in Proposition \ref{prop:Compactness of the Jacobian}. Then the rescaled energy density 
$\mu_{\eps}(t)$ \eqref{def:mu} concentrates at the vortex cores:  
\begin{equation}\label{eq:concentration of the energy density}
\lim_{\eps\to 0} \mu_{\eps}(t) = \pi\sum\limits_{k=1}^{N}\delta_{\xi_{k}(t)} \text{ in } (C_{c}^{0,1})^{*},
\end{equation}
for every $t\in [0,\bar{T}]$.
\end{proposition}
\begin{proof}
See \cite{CollianderGLStability}, Theorem~1.4.4. We denote the limiting measure $\pi\sum\limits_{k=1}^{N}\delta_{\xi_{k}(t)}$ by $\mu(t)$.
\qquad\end{proof}
\begin{proposition}[Concentration of the stress tensor]\label{prop:convergence of the stress tensor}
For every $t\in [0,\bar{T}]$ and for every smooth vector field $\vv:\, D\to \R^{2}$, there holds 
 \begin{equation}\label{eq:concentration of the stress tensor}
 \lim\limits_{\eps\to 0}\int\limits_{D}\vv\cdot k_{\eps}(\nabla u_{\eps}(t)\otimes \nabla u_{\eps}(t))\,dx= \pi\sum\limits_{k=1}^{N}\vv(\xi_{k}(t)).
\end{equation}
\end{proposition}
\begin{proof}
Due to Proposition \ref{prop:GLEnergyUpperBound}, the family $(k_{\eps}(\nabla u_{\eps}(t)\otimes \nabla u_{\eps}(t)))$ is relatively compact in the space $(C_{c}^{0,1})^{*}$. Moreover, due to Proposition \ref{prop:convergence of the energy density}, for every Borel set $A\subset D\setminus \{\xxi(t)\}$, there holds
\[\int\limits_{A}k_{\eps}|(\partial_{j}u_{\eps}(t), \partial_{l}u_{\eps}(t))|\,dx \leqslant \int\limits_{A} \mu_{\eps}(t) \,dx =o(1).
\]
Therefore, the stress tensor converges to the measure of the form $\sum_{k=1}^{N}A_{k}(t)\delta_{\xi_{k}(t)}$,
where $A_{k}(t)$ are some $2\times2$ matrices. It remains to show that $A_{k}(t)=\pi \Id$, for all $k$ and $t$. 
To this end, we use the result of Kurzke and Spirn on the equipartition of the Ginzburg-Landau energy (Proposition \ref{KurzkeSpirn}). 

We fix $t$, set  $\sigma=\tfrac{\rho_{\xxi(t)}}{3}$ and apply Proposition \ref{KurzkeSpirn} in $B_{\sigma}(\xi_{k}(t))$, for a $k\in\{1,..., N\}$. The assumptions are satisfied for all sufficiently small $\eps$, due to Propositions \ref{prop:Compactness of the Jacobian} and~\ref{prop:convergence of the energy density}. Moreover, we may take as $K_{0}$ the constant $C$ from Proposition \ref{prop:GLEnergyUpperBound}. Consequently, the constant is the same for all $t$ and $k$. We finally obtain
\[\Big| \int\limits_{B_{\sigma}}A_{k}(t)-\pi\Id dx\Big|\leqslant\lim\limits_{\eps\to0} \Big| \int\limits_{B_{\sigma}}k_{\eps}(\nabla u_{\eps}\otimes \nabla u_{\eps})-\pi\Id dx\Big|\leqslant \lim\limits_{\eps\to0}\frac{C}{\sqrt{\log(1/\eps)}}=0.
\]
\qquad\end{proof}

\begin{proposition}[Concentration of the momentum]\label{prop:convergence of the momentum}
For all $0\leqslant t_{1}<t_{2}\leqslant \bar{T}$ and all smooth vector fields $\w:D\times [0, \bar{T}]\rightarrow \R^{2}$, there holds
\begin{equation}\label{eq:concentration of the momentum}
 \lim\limits_{\eps\to 0}\int\limits_{t_{1}}^{t_{2}}\int\limits_{D} (k_{\eps}\pp(u_{\eps}(t)), \w) =-\pi\sum\limits_{k=1}^{N}\int\limits_{t_{1}}^{t_{2}}(\dot{\xi}_{k}(t),\w({\xi}_{k}(t),t))\, dt.
\end{equation}
\end{proposition}
\begin{proof}
The family $(k_{\eps}\pp(u_{\eps}))$ is precompact in the space $(C^{0,1}(D\times [0, \bar{T}]))^{*}$, grace to \eqref {kinetic_energy} and \eqref{GLEnergyUpperBound}. Moreover,  due to \eqref{eq:concentration of the energy density}, the limiting measure-valued vector field $\pp$ has the form
$\pp=\pi \sum_{k=1}^{N}\pp_{k}(t) \delta_{\xi_{k}(t)}$. The functions $\pp_{k}:[0, \bar{T}]\to \R^{2}$ satisfy
$\smallint_{0}^{\bar{T}}|\pp_{k}|\cdot| \dot{\xi}_{k}|\,dt<\infty$ for every $k\in \{1,..., N\}$.
It remains to prove that $\pp_{k}=-\pi\dot{\xi}_{k}$ almost everywhere. 

We observe that for any smooth function $\phi(x,t)$, 
there holds
\begin{equation}\label{DifferenceBetweenEnDensityAndFluidTensor}
 - \int\limits_{t_{1}}^{t_{2}}\int\limits_{D}\nabla\phi\, d\pp=\int\limits_{D} \phi(x, t_{2}) \, d\mu(t_{2})-\int\limits_{D} \phi(x, t_{1}) \, d\mu(t_{1})- \int\limits_{t_{1}}^{t_{2}}\int\limits_{D} \partial_{t}\phi(x,t) \, d\mu(t)\,dt.
\end{equation}

Indeed, we multiply \eqref{ConsOfEn} with $k_{\eps}\phi(x,t)$ and integrate over $D\times [t_{1},t_{2}]$. We obtain  
\begin{multline*}
\lambda_{\eps}k_{\eps}\int\limits_{t_{1}}^{t_{2}}\int\limits_{D}|\partial_{t}u_{\eps}|^{2}\phi +k_{\eps}^{2} \int\limits_{t_{1}}^{t_{2}}\int\limits_{D}(\F, \pp(u_{\eps}(t))) \phi -k_{\eps} \int\limits_{t_{1}}^{t_{2}}\int\limits_{D}
(\gG,\V(u_{\eps}))\phi  \\
=-k_{\eps}\int\limits_{t_{1}}^{t_{2}}\int\limits_{D} \partial_{t}e_{\eps}(u_{\eps})\phi  + k_{\eps}\int\limits_{t_{1}}^{t_{2}}\int\limits_{D} \diverg\pp(u_{\eps}) \phi. 
\end{multline*}
All terms in the first line converge to zero as $\eps$ goes to $0$. This follows from \eqref{kinetic_energy}, \eqref{EstOnF}, and \eqref{eq:EstimateOnVG}. We integrate by parts in the second line and arrive at \eqref{DifferenceBetweenEnDensityAndFluidTensor}.

We construct a special test function $\phi(x,t)$. For
$\rho_{0}:=\inf\{\rho_{\xxi(t)}, t\in [0,\bar{T}]\}>0
$, 
we set $\sigma=\rho_{0}/2$. We fix a smooth function $\vvarphi(x)$ such that 
$\vvarphi(x)=x $ in $B_{\sigma}(0)$ and $\vvarphi(x)=0$ outside of $B_{2\sigma}(0)$. We finally set $\pphi(x,t):=\vvarphi(x-\xi_{k}(t))$, for a fixed $k\in\{1,..., N\}$. When we consider \eqref{DifferenceBetweenEnDensityAndFluidTensor} with the test functions $\phi^{j}$, $j=1,2$, the identity takes the form
\[- \int_{t_{1}}^{t_{2}}p_{k}^{j}(t)\,dt=\pi\int_{t_{1}}^{t_{2}}\dot{\xi}^{j}_{k}(t) \, dt
.\]
The equality holds for all $[t_{1},t_{2}]\subset[0, \bar{T}]$. Thus $\pp_{k}=-\pi\dot{\xi}_{k}$ almost everywhere, as required. 
\qquad\end{proof}

\begin{proposition}
\label{DivCurrent}
The quantity $\diverg \jj(u_{\eps})$ converges to $0$ in the dual space of $W^{1,4}_{0}(D\times[0,\bar{T}])$.
\end{proposition}
\begin{proof}
We test the conservation law for the mass \eqref{ConsOfMass} with a function $\phi$ from $W^{1,4}_{0}(D\times[0,\bar{T}])$ and integrate over $D\times[0,\bar{T}]$. This gives 
\begin{multline}
-\int\limits_{0}^{\bar{T}}\int\limits_{D}(\nabla\phi,\jj(u_{\eps}))=\int\limits_{0}^{\bar{T}}\int\limits_{D}\phi\cdot \partial_{t}(\tfrac{1-|u_{\eps}|^{2}}{2})+k_{\eps}\int\limits_{0}^{\bar{T}}\int\limits_{D}\phi\cdot(\F,\jj(u_{\eps}))\nonumber\\
+\int\limits_{0}^{\bar{T}}\int\limits_{D}\phi\cdot(\gG, \nabla(\tfrac{1-|u_{\eps}|^{2}}{2}))+\lambda_{\eps}\int\limits_{0}^{\bar{T}}\int\limits_{D}\phi\cdot(u_{\eps}\times\partial_{t}u_{\eps}).
\end{multline}
We estimate the terms on the right-hand side using Propositions \ref{prop:kinetic_energy} and \ref{prop:GLEnergyUpperBound} together with the Sobolev embedding $W^{1,4}_{0}(D\times[0,\bar{T}])\hookrightarrow L^{\infty}(D\times[0,\bar{T}])$. We have
\begin{align*}
\Big|\int\limits_{0}^{\bar{T}}\int\limits_{D}\phi\partial_{t}(\tfrac{1-|u_{\eps}|^{2}}{2}) \Big|&=\Big|\int\limits_{0}^{\bar{T}}\int\limits_{D}\partial_{t}\phi\cdot (\tfrac{1-|u_{\eps}|^{2}}{2}) \Big|\leqslant C\lVert\partial_{t}\phi\rVert_{L^{2}}\eps\sqrt{\log(1/\eps)};
\end{align*}
\begin{multline*}
\Big|k_{\eps}\int\limits_{0}^{\bar{T}}\int\limits_{D}\phi(\F,\jj(u_{\eps}))\Big|\leqslant Ck_{\eps}\lVert\phi\rVert_{L^{\infty}}
\Big(\int\limits_{0}^{\bar{T}}\int\limits_{D}|u_{\eps}|^{2}\Big)^{\tfrac{1}{2}}\cdot\Big(\int\limits_{0}^{\bar{T}}\int\limits_{D}|
\nabla u_{\eps}|^{2}\Big)^{\tfrac{1}{2}}\\
\leqslant C\lVert\phi\rVert_{W^{1,4}_{0}}k_{\eps}^{1/2};
\end{multline*}
\begin{align*}
\Big|\int\limits_{0}^{\bar{T}}\int\limits_{D}\phi(\gG, \nabla(\tfrac{1-|u_{\eps}|^{2}}{2}))\Big|=\Big|\int\limits_{0}^{\bar{T}}\int\limits_{D}\nabla(\phi \gG)\cdot (\tfrac{1-|u_{\eps}|^{2}}{2}))\Big|
\leqslant C\lVert\nabla\phi\rVert_{W_{0}^{1,4}}\eps\sqrt{\log(1/\eps)};
\end{align*}
\begin{multline*}
\Big|\int\limits_{0}^{\bar{T}}\int\limits_{D}\phi(u_{\eps}\times\partial_{t}u_{\eps})\Big|\leqslant \lambda_{\eps}C\lVert\phi\rVert_{L^{\infty}}\Big(\int\limits_{0}^{\bar{T}}\int\limits_{D}|u_{\eps}|^{2}\Big)^{\tfrac{1}{2}}\cdot\Big(\int\limits_{0}^{\bar{T}}\int\limits_{D}|
\partial_{t} u_{\eps}|^{2}\Big)^{\tfrac{1}{2}}\\
\leqslant C\lVert\phi\rVert_{W^{1,4}_{0}}k_{\eps}^{1/2}.
\end{multline*}
All the quantities on the right-hand side are of the order $o(1)$, and so the claim follows.\quad\end{proof}

\section{Estimates involving the energy excess}\label{sec:EnExc}
We measure the difference between $\vortexvector (t)$ and $\xxi(t)$ on the level of the energy in terms of the energy excess. We now introduce two more quantities that we use for the same purpose.

The first of them is the \textit{vortex position error} $\eeta(t)$. It is a vector in $\R^{2N}$ with the two-dimensional components $\eta_{k}(t)=\xi_{k}(t)-a_{k}(t)$. The second quantity is the \textit{ODE error} $\rR(t)\in \R^{2N}$. We define it via the components as 
\begin{equation}\label{def:R}
R_{k}(t):=\lambda_{0}\dot{\xi}_{k}(t)+d_{k}i\dot{\xi}_{k}(t)+\tfrac{1}{\pi}\partial_{a_{k}}W(\vortexvector (t))-\F(a_{k}(t))-d_{k}i\gG(a_{k}(t)).
\end{equation}
In other words,
$R_{k}(t)=(\lambda_{0}+d_{k}i)\dot{\eta}_{k}(t).
$
Consequently, 
\begin{equation}\label{eq:RelationBetweenRandXi}
|\rR(t)|=\sqrt{\lambda_{0}^{2}+1}|\dot{\eeta}(t)|.
\end{equation}
In this section, we establish a relation between $D_{\eps}(t)$, $\eeta(t)$ and $\rR(t)$. It is given by the inequalities \eqref{GronwallOneDir} and \eqref{DiffFluidTensorAndCHM}.

We first need the following auxiliary estimate.
\begin{proposition}
For every $0\leqslant t_{1}<t_{2}\leqslant \bar{T}$, the energy excess satisfies the inequality
\begin{align}\label{EnExcAndPosDiff}
D_{\eps}(t_{2})-D_{\eps}(t_{1})\leqslant \, \pi\sum\limits_{k=1}^{N}\int\limits_{t_{1}}^{t_{2}}-2\lambda_{0} (\dot{a}_{k}, \dot{\eta}_{k}) +(\Q_{k},\dot{\eta}_{k})+(S_{k}\cdot \eta_{k}, \dot{a}_{k})\, dt + o(1).
\end{align}
Here, the functions 
$\Q_k(t):=\F(\xi_{k}(t),t)+d_{k}i\gG(\xi_{k}(t), t)
$ 
are vector-valued and the functions 
$S_{k}(t):=\smallint_{0}^{1}\nabla \F(a_{k}(t)+s\cdot\eta_{k}(t),t)+d_k\nabla (i\gG)(a_{k}(t)+s\cdot\eta_{k}(t),t)\, ds
$
are matrix-valued. 
\end{proposition}
\begin{proof}
By the definition of the energy excess, we have
\begin{align*}
D_{\eps}(t_{2})-D_{\eps}(t_{1})=E_{\eps}(u_{\eps}(t_{2}))-E_{\eps}(u_{\eps}(t_{1}))+W(\vortexvector (t_{1}))-W(\vortexvector (t_{2})).
\end{align*}
We rewrite the right-hand side with the help of \eqref{GLDetVortexMotion} and \eqref{eq:EnergyEvolution}. We obtain that 
\[D_{\eps}(t_{2})-D_{\eps}(t_{1})=B_{1}+B_{2}+B_{3}.
\]
The terms $B_{j}$ are given by
\[B_{1}=-\lambda_{\eps}\int\limits_{t_{1}}^{t_{2}}\int\limits_{D}|\partial_{t}u_{\eps}|^{2}+\pi \lambda_{0}\sum_{k=1}^{N}\int\limits_{t_{1}}^{t_{2}}|\dot{a}_{k}|^{2}\, dt,
\]
\[B_{2}=-k_{\eps}\int\limits_{t_{1}}^{t_{2}}\int\limits_{D}(\F,\pp(u_{\eps}))-\pi\sum_{k=1}^{N}\int\limits_{t_{1}}^{t_{2}}(\F(a_{k}(t),t), \dot{a}_{k})\,dt,
\]
and
\[B_{3}=\int\limits_{t_{1}}^{t_{2}}\int\limits_{D}(\gG,\V(u_{\eps}))-\pi\sum_{k=1}^{N}d_{k}\int \limits_{t_{1}}^{t_{2}}(i\gG(a_{k}(t),t),\dot{a}_{k})\,dt.
\]
We consider them one by one. 

Proposition \ref{prop: mobility bound} implies that
\begin{multline*}
B_{1}\leqslant \pi \lambda_{0}\sum_{k=1}^{N}\int\limits_{t_{1}}^{t_{2}}(|\dot{a}_{k}|^{2}-|\dot{\xi}_{k}|^{2})\, dt+o(1)
=\pi \lambda_{0}\sum_{k=1}^{N}\int\limits_{t_{1}}^{t_{2}}(-\dot{\eta}_{k}, \dot{\eta}_{k}+2\dot{a}_{k})\, dt +o(1)\\
\leqslant -2\pi\lambda_{0} \sum\limits_{k=1}^{N}\int\limits_{t_{1}}^{t_{2}}(\dot{a}_{k}, \dot{\eta}_{k})\, dt+o(1).
\end{multline*}
 
 The estimates for $B_{2}$ and $B_{3}$ follow from similar arguments. 
 By Proposition \ref{prop:convergence of the momentum}, we obtain 
 \begin{multline*}
 B_{2}=\pi\sum_{k=1}^{N}\int\limits_{t_{1}}^{t_{2}}(\F(\xi_{k}(t),t), \dot{\xi}_{k})-(\F(a_{k}(t),t), \dot{a}_{k})\,dt+o(1)\\
= \pi\sum_{k=1}^{N}\int\limits_{t_{1}}^{t_{2}}(\F(\xi_{k}(t),t),\dot{\eta}_{k})+(\F(\xi_{k}(t),t)-\F(a_{k}(t),t), \dot{a}_{k})\,dt+o(1).
 \end{multline*}
 We represent the vector $\F(\xi_{k})-\F(a_{k})$ with the help of the Mean Value Theorem~by
 \[\F(\xi_{k}(t),t)-\F(a_{k}(t),t)=(\textstyle \int_{0}^{1}\nabla \F(a_{k}(t)+s\cdot\eta_{k}(t),t)\,ds)\cdot\eta_{k}(t).
 \]
 
 For $B_{3}$, we use \eqref{eq:ConvOfVelocity} and the Mean Value Theorem~to get
 \begin{multline*}
 B_{3}=\pi\sum_{k=1}^{N}d_{k}\int\limits_{t_{1}}^{t_{2}}(i\gG(\xi_{k}(t),t),\dot{\eta}_{k})\,dt\\
+\pi\sum_{k=1}^{N}d_{k}\int\limits_{t_{1}}^{t_{2}}(\textstyle\int_{0}^{1}\nabla (i\gG)(a_{k}(t)+s\cdot\eta_{k}(t),t)\,ds)\cdot\eta_{k}(t), \dot{a}_{k})\, dt+o(1).
 \end{multline*}
 
We sum the estimates for $B_{j}$ and arrive at \eqref{EnExcAndPosDiff}.
\qquad\end{proof}

\begin{proposition}
For every $T<\bar{T}$, there exists a constant $C=C(T, \vortexvector^{0})$ such that, for $0\leqslant t_{1}< t_{2}\leqslant T$, there holds
\begin{equation}\label{GronwallOneDir}
D_{\eps}(t_{2})+|\eeta(t_{2})|\leqslant D_{\eps}(t_{1})+|\eeta(t_{1})|+C\int\limits_{t_{1}}^{t_{2}}|\eeta|\, dt +C\int\limits_{t_{1}}^{t_{2}}|\rR|\, dt+o(1).
\end{equation}
\end{proposition}

\begin{proof}
We estimate the right-hand side of \eqref {EnExcAndPosDiff} in terms of $\rR$. The time $T$ is strictly smaller than the collision time for the vortex motion law \eqref{GLDetVortexMotion}. Therefore, the right-hand side of \eqref{GLDetVortexMotion} is bounded on $[0, T]$. The functions $\Q_{k}$ and $S_{k}$ are bounded on this interval as well, because the fields $\F$ and $\gG$ are smooth. Hence, there exists a constant $C=(T, \vortexvector^{0})$ such that 
\[\sup\limits_{t\in[0, T]}|\dot{\vortexvector}(t)|\leqslant C,\, \max\limits_{k}\sup\limits_{t\in[0, T]}|\Q_{k}(t)| \leqslant C,  \text{ and
 }\max\limits_{k}\sup\limits_{t\in[0, T]}|S_{k}(t)| \leqslant C.\]

With this $C$ and $0\leqslant t_{1}< t_{2}\leqslant T$, the inequality \eqref{EnExcAndPosDiff} simplifies to
\[D_{\eps}(t_{2})\leqslant D_{\eps}(t_{1})+C\int\limits_{t_{1}}^{t_{2}}|\eeta(t)|\, dt +C\int\limits_{t_{1}}^{t_{2}}|\dot{\eeta}(t)|\, dt+o(1).
\]
We add to both sides of the inequality the term $|\eeta(t_{2})|-|\eeta(t_{1})|$. On the right-hand side, we estimate this term from above. We have 
$\big| |\eeta(t_{2})|-|\eeta(t_{1})|\big|\leqslant \int_{t_{1}}^{t_{2}}\big| \tfrac{d}{dt}|\eeta(t)|\big|\,dt.$
On the other hand, we have that
$\big|\tfrac{d}{dt}|\eeta(t)|\big|=|\dot{\eeta}(t)|.
$
This way, we see that
\[D_{\eps}(t_{2})+|\eeta(t_{2})|-|\eeta(t_{1})|\leqslant D_{\eps}(t_{1})+C\int\limits_{t_{1}}^{t_{2}}|\eeta(t)|\, dt +C\int\limits_{t_{1}}^{t_{2}}|\dot{\eeta}(t)|\, dt+o(1).
\]
This inequality together with \eqref{eq:RelationBetweenRandXi} yields \eqref{GronwallOneDir}. 
\qquad\end{proof}
 
The proposition below is the last major ingredient of the proof of Theorem~\ref{ThmGLDet}. 
The estimate has been proven in \cite{KMMS} with the tools developed in \cite{JerrardSpirnRefined}. The argument relies on the energy estimates, the concentration of the Jacobians, and the fact that $\diverg \jj(u_{\eps})$ converges to zero. Grace to Propositions \ref{prop:GLEnergyUpperBound}, \ref{prop:Compactness of the Jacobian}, and~\ref{DivCurrent}, the proof works in our case as well. Therefore, we do not repeat it.
\begin{proposition}[\cite{KMMS}, Proposition 5.4 ]\label{prop:DiffFluidTensorAndCHM}
There exists a constant $K>1$ depending only on $D$ and $N$ with the following property. Suppose that at a time $t_{0}\geqslant 0$, $D_{\eps}(t_{0})=o(1)$ and $\eeta(t_{0})=0$. Let $\tau>0$ be so small that
$k_{\eps}\smallint_{t_0}^{t_{0}+\tau}\smallint_{D}|\partial_{t}u_{\eps}|^{2}\leqslant C$. Set $\sigma:=\tfrac{1}{2}\rho_{\vortexvector(t_{0})}$. Then for every $\delta\leqslant \tau $ such that $D_{\eps}(t)\leqslant C$ and $|\eeta(t)|\leqslant \sigma/(4K)$ on the interval $[t_{0}, t_{0}+\delta]$, and for every $\varphi\in W_{0}^{2,\infty}(D)$ that is affine in each $B_{\sigma}(a_{k}(t_{0}))$, there holds 
\begin{multline}\label{DiffFluidTensorAndCHM}
\int\limits_{t_{0}}^{t_{0}+\delta}\Big|\int\limits_{D}( \grad \curl\varphi):\big ((\nabla u_{\eps}\otimes \nabla u_{\eps})-(\jj(u_{*})\otimes\jj(u_{*}))\big) \, dx\Big|\, dt\\
\leqslant C \int\limits_{t_{0}}^{t_{0}+\delta} D_{\eps}(t)+ |\eeta(t)|\,dt +o(1)
\end{multline}
with $u_{*}=u_{*}(\vortexvector(t),\dd)$.
\end{proposition}

\section{Proof of Theorem~\ref{ThmGLDet}}\label{proof_complete}

In this section, we finish the proof of Theorem~\ref{ThmGLDet}. We first establish a short-time result in Proposition \ref{prop:short-time}. With its help, we finally obtain the full statement on the time-interval $[0,T^{*})$.

\begin{proposition} \label{prop:short-time}
There exists a time $\tau>0$ such that the conclusions of Theorem~\ref{ThmGLDet} hold on $[0,\tau]$.
\end{proposition} 
\begin{proof}
%
The time $\bar{T}$ is strictly smaller than the time of the first collision for the solution to \eqref{GLDetVortexMotion}. Therefore, the number $A_{*}:=\sup\{\,|\dot{\vortexvector}(t)|, t\in [0, \bar{T}]\}$ is well-defined. We set $\rho_{*}:=\inf\{\rho_{\vortexvector(t)}, t\in[0, \bar{T}]\}$ and $\sigma:=\rho_{*}/2$. The radius $\sigma$ is strictly positive. For the constant $K>1$ from Proposition \ref{prop:DiffFluidTensorAndCHM}, we define 
$\tau_{0}:=\min\{\bar{T},\sigma/(4KA_{*})\}
$.
Then, $a_{k}(t)\in B_{\sigma/(4K)}(a_{k}^{0})$, for every $k\in\{1,...,N\}$ and all $t\in[0,\tau_{0}]$. Since  the paths $\xi_{k}(t)$ are continuous, there exists a time $\tau_{1}>0$ such that  $\xi_{k}(t)\in B_{\sigma}(a_{k}^{0})$, for every $k\in\{1,...,N\}$ and all $t\in[0,\tau_{1}]$. We finally set
\[\tau:=\min\{\tau_{0}, \tau_{1}\}.
\]
We show that Theorem~\ref{ThmGLDet} holds true on the time interval $[0,\tau]$. 

We fix a cut-off function $\chi_{\sigma}(x)\in C^{\infty}(\R^{2})$ such that $\chi(x)=1$ in $B_{\sigma}(0)$ and $\chi(x)=0$ outside of $B_{2\sigma}(0)$. We fix a number $k\in \{1,...,N\}$ and define a vector-valued test function $\pphi=(\phi^{1},\phi^{2})$ with
$\pphi(x)=(i x)\cdot \chi_{\sigma}(x-a_{k}^{0})
$.
In the ball $B_{\sigma}(a_{k}^{0})$, there holds
\[
\phi^{1}(x)=-x_{2}, \; \phi^{2}(x)=x_{1}, \; \curl \phi^{1}(x)=(1,0), \; \curl \phi^{2}(x)=(0,1).
\]
Outside of the ball $B_{2\sigma}(a_{k}^{0})$, all these functions are equal to zero. 

We apply Propositions \ref{prop:Compactness of the Jacobian} and \ref{prop:convergence of the momentum} to the test functions $\phi^{j}$, $j=1,2$. We obtain
\[\int\limits_{0}^{\tau}\partial_{t}\int\limits_{D}J(u_{\eps}(t))\phi^{j}\,dx dt-\lambda_{\eps}\int\limits_{0}^{\tau}\int\limits_{D}(\curl \phi^{j},\pp (u_{\eps}))\,dx dt
\rightarrow \pi\int\limits_{0}^{\tau}\partial_{t}(d_{k}(i\xi_{k})^{j}+\lambda_{0}\xi_{k}^{j})\, dt.
\]
For the equation error 
\[R_{k}^{j}=\lambda_{0}\dot{\xi}_{k}^{j}+d_{k}i\dot{\xi}_{k}^{j}+\frac{1}{\pi}\partial_{a_{k}}W(\vortexvector(t))^{j}-F^{j}(a_{k}(t), t)-d_{k}(i\gG)^{j}(a_{k}(t), t),\]
there holds 
\begin{multline}\label{Estimate on equation error general form}
\pi\int\limits_{0}^{\tau}| R_{k}^{j}|\,dt
\leqslant \lim_{\eps\to0}\int\limits_{0}^{\tau}\Big|\partial_{t}\int\limits_{D}J(u_{\eps}(t))\phi^{j}\,dx -\lambda_{\eps}\int\limits_{D}(\curl \phi^{j},\pp (u_{\eps}))\,dx\\
+\partial_{a_{k}}W(\vortexvector(t))^{j}-\pi F^{j}(a_{k}(t), t)-\pi d_{k}(i\gG)^{j}(a_{k}(t), t)\Big|\,dt.
\end{multline}
We substitute the first two terms on the right-hand side of \eqref{Estimate on equation error general form} from \eqref{ConsOfJacobian}.
This yields the inequality
\[\pi\int\limits_{0}^{\tau}| R_{k}^{j}|\,dt\leqslant \lim_{\eps\to 0} (I_{1}+I_{2}+I_{3}),
\]
where the terms $I_{1}$, $I_{2}$, $I_{3}$ are given by
\begin{equation*}
I_{1}=\int\limits_{0}^{\tau}\Big|\int\limits_{D}\phi^{j} \cdot \curl \diverg (\nabla u_{\eps}\otimes \nabla u_{\eps})\,dx+\partial_{a_{k}}W(\vortexvector(t))^{j}\Big|\,dt,
\end{equation*}

\begin{equation*}
I_{2}=\int\limits_{0}^{\tau}\Big|\int\limits_{D}k_{\eps}(\curl \phi^{j},\F\cdot  (\nabla u_{\eps}\otimes \nabla u_{\eps}))\,dx-\pi F^{j}(a_{k}(t), t)\Big|\,dt,
\end{equation*}

\begin{equation*}
I_{3}=\int\limits_{0}^{\tau}\Big|\int\limits_{D}(\curl \phi^{j}, i\gG\cdot J( u_{\eps}))\,dx-\pi d_{k}(i\gG)^{j}(a_{k}(t), t)\Big|\,dt.
\end{equation*}

We estimate them one by one. 

By the choice of $\pphi$, 
$\partial_{a_{k}}W(\vortexvector(t))^{j}=\curl\phi^{j}\cdot \partial_{a_{k}}W(\vortexvector(t))
$
in $B_{\sigma}(a_{k}^{0})$.
Therefore, by Propositions \ref{prop:GradOfW} and \ref{prop:DiffFluidTensorAndCHM}, we have
\begin{multline*}
I_{1}=\int\limits_{0}^{\tau}\Big|\int\limits_{D}(\grad \curl \phi^{j}):((\nabla u_{\eps}\otimes \nabla u_{\eps})-(\jj(u_{*}(\vortexvector))\otimes \jj(u_{*}(\vortexvector))))\Big|\\
\leqslant C \int\limits_{0}^{\tau} D_{\eps}(t)+ |\eeta(t)| +o(1).
\end{multline*}
For the term $I_{2}$ we obtain, by Proposition \ref{prop:convergence of the stress tensor} and the smoothness of $\F$, that
\[I_{2}=\pi\int\limits_{0}^{\tau}\Big|F^{j}(\xi_{k}(t), t)-F^{j}(a_{k}(t),t)\Big|+o(1)\leqslant C\int\limits_{0}^{\tau}|\eta_{k}|+o(1).
\]

Analogously, by Proposition \ref{prop:Compactness of the Jacobian} we have that
\[I_{3}=\pi\int\limits_{0}^{\tau}\Big|(i\gG)^{j}(\xi_{k}(t),t)-(i\gG)^{j}(a_{k}(t),t)\Big|\,dt+o(1)\leqslant C\int\limits_{0}^{\tau}|\eta_{k}|\, dt+o(1).
\]
We sum over $j$ and $k$ to see that
\[\int\limits_{0}^{\tau}| \rR|\,dt\leqslant \lim_{\eps\to0}C\int\limits_{0}^{\tau}|\eeta(t)|+D_{\eps}(t)\, dt.
\]
We combine this inequality with \eqref{GronwallOneDir}. At $t=0$, $\eeta(0)=0$ and $D_{\eps}(0)=o(1)$. With the Gronwall Lemma we conclude that $\eeta(t)=0$ and $D_{\eps}(t)=o(1)$ for all $t\in[0,\tau]$. The first identity establishes the vortex motion law. The second shows that the solution remains well-prepared on $[0,\tau]$.
\qquad\end{proof}

\begin{proof}[Proof of Theorem~\ref{ThmGLDet}, completed]
We iterate our argument starting at the time $\tau$, which is found in Proposition \ref{prop:short-time}. We apply successively Propositions \ref{prop:kinetic_energy} -- \ref{prop:short-time} on a new time-interval $[\tau, \tau+\bar{\tau}]$ with $\bar{\tau}>0$. It remains to check that we can reach, by iterating so, the terminal time $T^{*}$. 

We argue by contradiction. Let $\tilde {T}$ be the supremum over all times for which Theorem~\ref
{ThmGLDet} holds. By Proposition \ref{prop:short-time}, $\tilde {T}$ is positive. Suppose that $\tilde {T}<T^{*}$. We shall show that we can actually apply Proposition \ref{prop:short-time} at the time $\tilde {T}$. This will contradict the assumption on maximality of $\tilde {T}$. 

We have to prove that $\xi_{k}(\tilde {T})=a_{k}(\tilde {T})$ and $D_{\eps}(\tilde {T})=o(1)$. For every $T\in [0, \tilde {T})$,  \eqref{eq:EnergyEvolution} gives
\begin{equation}\label{eq:EnergyFinal}
\lambda_{\eps}\int\limits_{0}^{T}\int\limits_{D}|\partial_{t}u_{\eps}|^{2}=E_{\eps}(u_{\eps}^{0})-E_{\eps}(u_{\eps}(T))
-k_{\eps}\int\limits_{0}^{T}\int\limits_{D}(\F, \pp(u_{\eps}))
+\int\limits_{0}^{T}\int\limits_{D}(\gG, \V(u_{\eps})).
\end{equation}
Using Propositions \ref{prop:short-time}, \ref{prop:convergence of the momentum}, and \ref{prop:Compactness of the Jacobian}, we see that the right-hand side is equal to 
\[ W(\vortexvector^{0})-W(\vortexvector(T))-\pi\sum\limits_{k=1}^{N}\int\limits_{0}^{T}(\F(a_{k}(t), t),\dot{a}_{k})-\pi\sum\limits_{k=1}^{N}d_{k}\int\limits_{0}^{T}(i\gG(a_{k}(t), t), \dot{a}_{k})+o(1).
\]
The paths $a_{k}$ remain distinct for $T=\tilde{T}$ because $\tilde{T}<T^{*}$. Thus, the quantity above remains bounded. This means that
$\lambda_{\eps}\smallint_{0}^{\tilde{T}}\smallint_{D}|\partial_{t}u_{\eps}|^{2}\leqslant C$
.
By Theorem~3 of \cite{SSProd}, the functions $\xi_{k}$ are continuous on the \textit{closed} interval $[0,\tilde{T}]$. Since $\vortexvector(T)$ is equal to $\xxi(T)$ on $[0, \tilde{T})$, and both collections of functions are continuous, we obtain $\vortexvector(\tilde{T})=\xxi(\tilde{T})$.

We now consider \eqref{eq:EnergyFinal} on the interval $[T, \tilde T]$ instead of $[0, T]$. Estimating the integrals with $\F$ and $\gG$ as before, we get
\[E_{\eps}(u_{\eps}(\tilde{T}))\leqslant E_{\eps}(u_{\eps}(T))-\pi\sum\limits_{k=1}^{N}\int\limits_{0}^{T}(\F(a_{k}(t), t)+d_{k}i\gG(a_{k}(t), t), \dot{a}_{k})+o(1).\] 
By the well-preparedness of $(u_{\eps})$ at all times $T<\tilde{T}$, we have that
\[E_{\eps}(u_{\eps}(T))=\pi N \log(1/\eps)+N\gamma+W(\vortexvector(T))+o(1).
\]
For each $k\in\{1,..., N\}$, we have by \eqref{eq:mobility bound} and the regularity of $\F$ and $\gG$ that
\[
\big|\int\limits_{T}^{\tilde{T}}(\F(a_{k})+d_{k}i\gG(a_{k}), \dot{a}_{k})\big|
\leqslant C\big(\int\limits_{T}^{\tilde{T}}|\gG(a_{k})|^{2}+|\F(a_{k})|^{2}\big)^{\tfrac{1}{2}}\cdot\big(\int\limits_{T}^{\tilde{T}}|\dot{a}_{k}|^{2}\big)^{\tfrac{1}{2}}
= C(\tilde{T}-T)^{\tfrac{1}{2}}.
\]

This gives 
\[E_{\eps}(u_{\eps}(\tilde{T}))\leqslant \pi N \log(1/\eps)+N\gamma+W(\vortexvector(T))+C(\tilde{T}-T)^{\tfrac{1}{2}}+o(1).
\]
Since $W(\vortexvector(T))\to W(\vortexvector(\tilde{T}))$ for $T\to \tilde{T}$, we obtain that
\[\lim\limits_{\eps\to 0}D_{\eps}(\tilde{T})\leqslant \lim\limits_{\eps\to 0}\lim\limits_{T\to \tilde{T}}(W(\vortexvector(T))+C(\tilde{T}-T)^{\tfrac{1}{2}}+o(1)-W(\vortexvector(\tilde{T})))=0.
\]
The proof is now complete.
\qquad\end{proof}

\begin{appendix}
\section{Equipartition of the Ginzburg-Landau energy}
\begin{proposition}[\cite{KurzkeSpirnEquipartition}, Theorem~1 and \cite{KurzkeVortexLiquids}, Proposition 15]\label{KurzkeSpirn}
Suppose that for some $\sigma>0$, we have
\[\lVert J(u_{\eps})-\pi\delta_{0}\rVert_{\left(C_{c}^{0,1}(B_{\sigma})\right)^{*}}\leqslant \frac{\sigma}{4}
\]
and 
\[\int\limits_{B_{\sigma}}e_{\eps}(u_{\eps})\, dx\leqslant \pi\log\frac{\sigma}{\eps}+K_{0}.
\]
Then there holds 
$\Big| \int\limits_{B_{\sigma}}k_{\eps}(\nabla u_{\eps}\otimes \nabla u_{\eps})-\pi\Id dx\Big|\leqslant K_{1}\sqrt{k_{\eps}}$,
with $K_{1}$ depending continuously on $K_{0}$.  
\end{proposition}
\end{appendix}

\textbf{Acknowlegements.} This paper is part of the author\rq{}s PhD Thesis. The author wishes to thank her advisor Professor Christof Melcher for pointing out the problem and many fruitful discussions. The author was partially supported by the RWTH Research Fellowship Program Russia.

\end{document}